\documentclass[12pt,a4paper,twoside,english]{amsart}

\usepackage[english]{babel}
\usepackage[latin1]{inputenc}
\usepackage{amssymb, amsmath, amsrefs,amsthm}
\usepackage{graphicx}
\usepackage{hyperref}
\usepackage{enumitem}
\usepackage{color}

\newcommand{\norm}[1]{\left\Vert#1\right\Vert}
\newcommand{\brkt}[1]{\left(#1\right)}

\newcommand{\abs}[1]{\left|#1\right|}

\newtheorem{thm}{Theorem}[section]
\newtheorem{cor}[thm]{Corollary}
\newtheorem{ex}[thm]{Example}
\newtheorem{lem}[thm]{Lemma}
\newtheorem{prop}[thm]{Proposition}
\newtheorem{rem}[thm]{Remark}
\newtheorem{defi}[thm]{Definition}

\usepackage{float}

 \makeatletter

\renewcommand\subsubsection{\@secnumfont}{\bfseries}%
\renewcommand\subsubsection{\@startsection{subsubsection}{3}
   \z@{.5\linespacing\@plus.7\linespacing}{-.5em}%
   {\normalfont\bfseries}}
\makeatother

\newcommand{\R}{\mathbb{R}}

\newcommand{\Z}{\mathbb{Z}}

\newcommand{\N}{\mathbb{N}}
\newcommand{\SW}{\mathcal{S}}
\newcommand{\BMO}{\mathrm{BMO}}
\newcommand{\bmo}{\mathrm{bmo}}
\newcommand{\ddd}{\,\text{\rm{\mbox{\dj}}}}
\newcommand{\dd}{\mathrm{d}}
\newcommand{\m}{\mathtt{m}}

\title[] {Endpoint estimates for bilinear pseudodifferential operators with symbol in $BS_{1,1}^m$}
\author[S. Arias]{Sergi Arias}
\address{Department of Mathematics, Stockholm University, SE-106 91 Stockholm, Sweden}
\email{arias@math.su.se}
\author[S.~Rodr\'iguez-L\'opez]{Salvador Rodr\'iguez-L\'opez}
\address{Department of Mathematics, Stockholm University, SE-106 91 Stockholm, Sweden}
\email{s.rodriguez-lopez@math.su.se}

\thanks{The authors are partially supported by the Spanish Government grant PID2020-113048GB-I00}

\makeatletter
\@namedef{subjclassname@2020}{%
  \textup{2020} Mathematics Subject Classification}
\makeatother


\begin{document}

\begin{abstract}
    In this paper we establish some endpoint estimates for bilinear pseudodifferential operators with symbol in the class $BS_{1,1}^m$, involving the space of functions with local bounded mean oscillation $\bmo(\R^n)$. As a consequence we also obtain an endpoint estimate of Kato-Ponce type.
\end{abstract}
\maketitle

\section{Introduction}
The present paper is devoted to the study of endpoint estimates for bilinear pseudodifferential operators of the form
\[
    T_\sigma(f,g)(x)=\iint\sigma(x,\xi,\eta)\widehat{f}(\xi)\widehat{g}(\eta)e^{ix(\xi+\eta)}\ddd \xi\ddd\eta,\quad x\in\R^n,\quad f,g\in\SW(\R^n),
\]
where the symbol $\sigma(x,\xi,\eta)$  belongs to the class $BS_{1,1}^m=BS_{1,1}^m(\R^n)$ with $m\in\R$, that is, $\sigma$ is a smooth function in $\R^{3n}$ for which
\[
    \abs{\partial^\alpha_x\partial^\beta_\xi\partial^\gamma_\eta\sigma(x,\xi,\eta)}\leq C_{\alpha,\beta,\gamma}(1+\abs{\xi}+\abs{\eta})^{m+\abs{\alpha}-\abs{\beta}-\abs{\gamma}}
\]
for all $(x,\xi,\eta)\in\R^{3n}$, all multi-indices $\alpha,\beta,\gamma\in\N^n$ and some $C_{\alpha,\beta,\gamma}>0$. Here $\ddd \xi$ denotes the normalised Lebesgue measure $\ddd \xi=(2\pi)^{-n}\dd\xi$.

For symbols in the class $BS^m_{1,1}$ and $N\in\N$ we shall use the notation
\begin{align*}
    \norm{\sigma}_{BS^m_{1,1;N}}:=\max_{\abs{\alpha},\abs{\beta},\abs{\gamma}\leq N}\left(\sup_{x,\xi,\eta\in\R^n}(1+\abs{\xi}+\abs{\eta})^{-(m+\abs{\alpha}-\abs{\beta}-\abs{\gamma})}\abs{\partial^\alpha_x\partial^\beta_\xi\partial^\gamma_\eta\sigma(x,\xi,\eta)}\right).
\end{align*}
\'A. B\'enyi and R.H. Torres showed in \cite{ben-tor}*{Theorem 2} that for all symbols $\sigma\in BS^0_{1,1}$ the associated pseudodifferential operator $T_\sigma$ satisfies the estimate
\begin{equation}\label{eqLps}
    \norm{T_\sigma(f,g)}_{L^r_s(\R^n)}\lesssim\norm{\sigma}_{BS_{1,1;N}^0}\left(\norm{f}_{L^p_s(\R^n)}\norm{g}_{L^q(\R^n)}+\norm{f}_{L^p(\R^n)}\norm{g}_{L^q_s(\R^n)}\right)
\end{equation}
for $N$ large enough, where $s>0$, $1<p,q,r<\infty$ and $1/p+1/q=1/r$. Here $L^p_s(\R^n)$ refers to Sobolev spaces (see Section \ref{prel} for the precise definition of the function spaces appearing in the paper).

\'A. B\'enyi, A. Nahmod and R. H. Torres obtained a similar estimate for $\sigma\in BS_{1,1}^m$ with $m\geq 0$ in \cite{ben-nah-tor}*{Theorem 2.1}, while V. Naibo \cite{Naibo}*{Section 1.1.1} studied the boundedness of $T_\sigma$ for $\sigma\in BS_{1,1}^0$ on Triebel-Lizorklin and Besov spaces.

Later, K. Koezuka and N. Tomita \cite{koe-tom} generalised the estimate in \eqref{eqLps} for symbols in $BS^m_{1,1}$, $m\in\R$, considering Triebel-Lizorkin and local Hardy spaces. More precisely, the authors show in \cite{koe-tom}*{Theorem 1.1} the existence of a positive integer $N$ for which
\begin{equation}\label{eq_ko_to}
    \norm{T_\sigma(f,g)}_{F_{p,q}^s(\R^n)}\lesssim\norm{\sigma}_{BS_{1,1;N}^m}\left(\norm{f}_{F_{p_1,q}^{s+m}(\R^n)}\norm{g}_{h^{p_2}(\R^n)}+\norm{f}_{h^{\tilde{p}_1}(\R^n)}\norm{g}_{F_{\tilde{p}_2,q}^{s+m}(\R^n)}\right)
\end{equation}
with $0<p_1,p_2,\tilde{p}_1,\tilde{p}_2,p<\infty$, $1/p_1+1/p_2=1/\tilde{p}_1+1/\tilde{p}_2=1/p$, $0<q\leq\infty$ and
\[
    s>\left\{
        \begin{array}{lll}
        \tau_{p,q} & \mathrm{if} & 0<q<\infty,\\
        \tau_{p,q}+n & \mathrm{if} & q=\infty, 0<p<1,\\
        n/p & \mathrm{if} & q=\infty, 1\leq p\leq\infty,
        \end{array} 
    \right.
\]
where 
\[
    \tau_{p,q}:=n\brkt{\frac{1}{\min(1,p,q)}-1}.
\]
In addition, it is also shown that \eqref{eq_ko_to} is satisfied if $p_2=\infty$ or $\tilde{p}_1=\infty$, with $L^\infty(\R^n)$ instead of the  local Hardy space.

Recently, B. J. Park \cite{Park} obtained estimates for multilinear pseudodifferential operators with symbol in the multilinear anlogue of $BS^m_{1,1}$, improving the result of K. Koezuka and N. Tomita. More precisely, it is shown in \cite{Park}*{Theorem 1.5} that \eqref{eq_ko_to} holds for $s>\tau_{p,q}$ and including the cases $p_1=\infty$ and $\tilde{p}_2=\infty$.

One of the goals of this paper is to obtain an inequality analogous to \eqref{eq_ko_to} at the endpoint case $p_2=\tilde{p}_1=\infty$, replacing $L^\infty(\R^n)$ by the larger space $\bmo(\R^n)$, the space of functions with local bounded mean oscillation. In Corollary \ref{cor_BS_dia} we present such estimate where the target space is then of Triebel-Lizorkin type with logarithmic smoothness (see Definition \ref{GSTL}).

The boundedness properties of bilinear pseudodifferential operators with symbols in $BS_{1,1}^m$, can be established by studying those in a subclass of symbols of the form
\begin{equation}\label{main_sym}
    \sigma(x,\xi,\eta)=\sum_{j=0}^\infty \m_j(x)\psi_j(\xi)\phi_j(\eta),
\end{equation}
where $\lbrace \m_j\rbrace_{j=0}^\infty$,$\lbrace \psi_j\rbrace_{j=0}^\infty$ and $\lbrace \phi_j\rbrace_{j=0}^\infty$ are collections of suitable smooth functions in $\R^n$. More specifically, we assume that for every $N\in\N$ there exists $C_N>0$ such that
\begin{equation}\label{equ4}
    \norm{\partial^\alpha\m_j}_\infty\leq C_N 2^{j(m+\abs{\alpha})}
\end{equation}
for all $\abs{\alpha}\leq N$ and some $m\in\R$, as well as
    \begin{align}\label{equ11}
        &\operatorname{supp}\psi_0\subseteq\lbrace\abs{\xi}\lesssim 1\rbrace,\quad \operatorname{supp}\psi_j\subseteq\lbrace\abs{\xi}\approx 2^j\rbrace \quad \mathrm{for} \quad j\geq 1,\nonumber\\
        &\norm{\partial^\alpha\psi_j}_\infty\leq C_N 2^{-j\abs{\alpha}} \quad \mathrm{for \ all} \quad \abs{\alpha}\leq N,
    \end{align}
    and
    \begin{align}\label{equ6}
        &\operatorname{supp}\phi_j\subseteq\lbrace\abs{\xi}\lesssim 2^j\rbrace \quad \mathrm{for} \quad j\geq 0,\nonumber\\
        &\norm{\partial^\alpha\phi_j}_\infty\leq C_N 2^{-j\abs{\alpha}} \quad \mathrm{for \ all} \quad \abs{\alpha}\leq N.
    \end{align}
Additionally, we suppose that
\begin{equation}\label{equ35}
    \phi_j=\sum_{k=0}^j\tilde{\psi}_k,\quad j\geq 0,
\end{equation}
where $\lbrace\tilde{\psi_j}\rbrace_{j=0}^\infty$ is a collection of smooth functions that satisfies the same support conditions as $\lbrace \psi_j\rbrace_{j=0}^\infty$.

The main goal of this paper is then to establish endpoint estimates for bilinear pseudodifferential operators with symbols as those described in \eqref{main_sym}.  More precisely, given a symbol $\sigma$ as in \eqref{main_sym}, we study the case when one of the arguments of the operator $T_\sigma(f,g)$ belongs to a space of Triebel-Lizorkin type and the other one to $\bmo(\R^n)$. We shall notice that $f$ and $g$ are affected differently when applying $T_\sigma$ since they are localised, in frequency, in rings or balls. In our main result (Theorem \ref{main_thm}) we distinguish the following two cases: 
\begin{enumerate}[label=(\roman*)]
\item\label{case1}When $f$ belongs to a Triebel-Lizorkin space and $g$ to a function space lying in between $L^\infty(\R^n)$ and $\bmo(\R^n)$, denoted by $X_w(\R^n)$ and described with the aid of a weight function $w$ (see Definition \ref{X_w}).

\item\label{case2} When $f$ belongs to a Besov space of the type $B_{\infty,\infty}^s(\R^n)$ and $g$ belongs to a Triebel-Lizorkin space of generalised smoothness.
\end{enumerate}
In both situations \ref{case1} and \ref{case2} we show that the function $T_\sigma(f,g)$ lies in a Triebel-Lizorkin space of generalised smoothness. We shall also point out that $\bmo(\R^n)$ is continuously embedded in $B_{\infty,\infty}^0(\R^n)$ (see Remark \ref{emb_Xw_besov}).

In Corollary \ref{cor_BS}, we extend the result in Theorem \ref{main_thm} to bilinear pseudodifferential operators with general symbol in $BS_ {1,1}^m$.

As an additional consequence of our study, we obtain some Kato-Ponce type estimates. In a previous work  \cite{Paper1} we obtained a result of this kind involving the endpoint space $\bmo(\R^n)$. More precisely, we show in \cite{Paper1}*{Corollary 6.1} that if $s>4n+1$ and $1<p<\infty$ then
\begin{equation}\label{eqpaper1}
    \norm{J^s(fg)}_{F_{p,2}^{0,1/(1+\log_+1/t)}(\R^n)}\lesssim\norm{J^sf}_{L^p(\R^n)}\norm{g}_{\bmo(\R^n)}+\norm{f}_{L^p(\R^n)}\norm{J^sg}_{\bmo(\R^n)},
\end{equation}
where $\log_+t=\log t$ if $t>1$ and $\log_+t=1$ if $0<t\leq 1$. Some of the questions arising from that result were whether it could be extended to $s>0$ or if a similar estimate holds when $0<p\leq 1$ or when $p=\infty$. In the present paper we extend that result twofold. In Corollary \ref{cor_kp} we show that the estimate in \eqref{eqpaper1} holds for $s>0$ and we are able to prove the same property for $0<p\leq 1$ replacing $L^p(\R^n)$ by $h^p(\R^n)$, the local Hardy space, and for $p=\infty$. To be more concrete, we provide a Kato-Ponce inequality when one of the terms lie in $F_{p,q}^0(\R^n)$, a space of Triebel-Lizorkin type, and the other in $\bmo(\R^n)$, when $0<p\leq\infty$, $1<q<\infty$ and $s>\tau_{p,q}$.

These kind of inequalities have been largely studied. For further results on this topic we also refer to the works by L. Grafakos and S. Oh \cite{Grafakos-Oh}, by L. Grafakos, D. Maldonado and V. Naibo \cite{Gra-Mal_Nai}, V. Naibo and A. Thomson \cite{Naibo-Thomson}, J. Brummer and V. Naibo \cites{Bru-Nai1,Bru-Nai2} or K. Koezuka and N. Tomita \cite{koe-tom}.

The paper is organised as follows. In Section \ref{prel} we recall the definition of the main function spaces and mathematical objects that we will be working with, as well as some general lemmas that we shall make use of. In Section \ref{main_results} we state the principal result obtained in this article and prove the main consequences. Finally, we give in Section \ref{Proof_main_thm} the proof of the main theorem. 

\section{Preliminaries}\label{prel}
The notation $A\lesssim B$ will be used to indicate the existence of a constant $C>0$ such that $A\leq C B$. Similarly, we will write $A\thickapprox B$ if both $A\lesssim B$ and $B\lesssim A$ hold. We will also use the notation $\langle\xi\rangle=(1+\vert\xi\vert^2)^{1/2}$ for $\xi\in\R^n$.

The space of Schwartz functions will be denoted by $\SW(\R^n)$ and its topological dual, the space of tempered distributions, by $\SW'(\R^n)$. For a function $f\in \SW(\R^n)$ we define its Fourier transform as
\[
    \mathcal{F}[f](\xi)=\widehat{f}(\xi)=\int_{\R^n}f(x)e^{-i x\xi}dx
\]
and we will write
\[
    a(tD)f(x)=\int_{\R^n}a(t\xi)\widehat{f}(\xi)e^{i x\xi}\ddd\xi
\]
for appropriate symbols $a$, or simply $a(D)$ when $t=1$.

\subsection{Function spaces} Let us recall the definition of some spaces of functions that appear in our study.

The Hardy space $H^1(\R^n)$ is defined as the set of tempered distributions $f$ for which the non-tangential maximal function
\begin{equation*}
    x\mapsto\sup_{t>0}\sup_{\abs{x-y}<t}\abs{(\Phi_t\ast f)(y)}
\end{equation*}
belongs to $L^1(\R^n)$, endowed with the norm
\begin{equation*}
    \norm{f}_{H^1(\R^n)}:=\norm{\sup_{t>0}\sup_{\abs{x-y}<t}\abs{(\Phi_t\ast f)(y)}}_{L^1(\R^n)}.
\end{equation*}
Here $\Phi$ is a Schwartz function with $\int\Phi=1$ and $\Phi_t(x)=t^{-n}\Phi(x/t)$, with $t>0$ and $x\in\R^n$. 

The local version of the Hardy space, denoted by $h^1(\R^n)$ and introduced by D. Goldberg \cite{Goldberg}, is the space of tempered distributions $f$ for which the truncated non-tangential maximal function
\begin{equation}\label{equ16}
    x\mapsto\sup_{0<t<\frac{1}{2}}\sup_{\abs{x-y}<t}\abs{(\Phi_t\ast f)(y)}
\end{equation}
belongs to $L^1(\R^n)$, endowed with the norm given by 
\begin{equation}\label{equ17}
    \norm{f}_{h^1(\R^n)}:=\norm{\sup_{0<t<\frac{1}{2}}\sup_{\abs{x-y}<t}\abs{(\Phi_t\ast f)(y)}}_{L^1(\R^n)}.
\end{equation}
The local Hardy spaces $h^p(\R^n)$, with $0<p<\infty$, are defined analogously to $h^1(\R^n)$, where it is now required that the function in \eqref{equ16} belongs to $L^p(\R^n)$ instead of $L^1(\R^n)$. The norm $\norm{\cdot}_{h^p(\R^n)}$ is defined as in \eqref{equ17} replacing $L^1(\R^n)$ by $L^p(\R^n)$. It holds that $h^p(\R^n)=L^p(\R^n)$ for all $1<p<\infty$ (see \cite{Goldberg}*{after Proposition 2}).

The space of functions with bounded mean oscillation, $\BMO(\R^n)$, which can be identified with the dual space of $H^1(\R^n)$, is the set of all those locally integrable functions $f$ on $\R^n$ for which
\[
    \norm{f}_{\BMO(\R^n)}:=\sup_Q\frac{1}{\vert Q\vert}\int_Q\vert f(x)-f_Q\vert\dd x<\infty.
\]
The supremum is taken over all cubes in $\R^n$ whose sides are parallel to the axis, while $\vert Q\vert$ denotes the Lebesgue measure of the cube $Q$ and $f_Q$ is the average of $f$ over $Q$, namely $f_Q=\frac{1}{\vert Q\vert}\int_Qf(x)\dd x$.

The local version of $\BMO(\R^n)$ was considered by D. Goldberg in \cite{Goldberg}, and it will be denoted by $\bmo(\R^n)$. It is defined to be the set of all locally integrable functions $f$ on $\R^n$ for which
\begin{equation*}
    \norm{f}_{\bmo(\R^n)}:=\sup_{\ell(Q)<1}\frac{1}{\vert Q\vert}\int_Q\vert f(x)-f_Q\vert\dd x + \sup_{\ell(Q)\geq 1}\frac{1}{\vert Q\vert}\int_Q\vert f(x)\vert\dd x<\infty.
\end{equation*}
Here $\ell(Q)$ denotes the side length of the cube $Q$. The function space $\bmo(\R^n)$ is the dual space of $h^1(\R^n)$ and it is continuously embedded in $\BMO(\R^n)$.

The following classical result for functions of bounded mean oscillation will be used later.

\begin{prop}\label{QofBMO}
    \cite{Stein}*{p.161} Let $\psi$ be a smooth function supported in a ring and satisfying $\int\psi=0$. There exists $C>0$ such that
    \[
        \norm{\psi(tD)f}_\infty\leq C\norm{f}_\BMO
    \]
    for all $f\in\BMO(\R^n)$ and $t>0$.
\end{prop}

Finally, we recall the definition of (inhomogeneous) Sobolev spaces. Given $s\in\R$ and $1<p<\infty$ we define $L^p_s(\R^n)$ as the set of tempered distributions for which
\[
    \norm{f}_{L^p_s(\R^n)}:=\norm{J^sf}_{L^p(\R^n)}<\infty,
\]
where $J^sf=\mathcal{F}^{-1}[\langle\xi\rangle^s\widehat{f}]$.
\subsection{Admissible weights}
The endpoint estimates that we investigate in this article involve the following spaces, defined using a certain type of weights, which were developed previously in \cite{rod-sta} and \cite{Paper1}.
\begin{defi}\label{X_w}
    Let $w:(0,\infty)\rightarrow (0,\infty)$ be a function satisfying the following properties:
        \begin{enumerate}[label=\Roman*)]
            \item \label{I}For every {compact} interval $I\subseteq(0,\infty)$ we have that
                  \[
                    0<\inf_{t\in I}\left(\inf_{s>0}\frac{w(st)}{w(s)}\right)\leq\sup_{t\in I}\left(\sup_{s>0}\frac{w(st)}{w(s)}\right)<\infty;
                  \]
            \item \label{II} There exists $N>0$ such that $\sup_{t>0}w(t)(1+1/t)^{-N}<\infty$;
            \item \label{III} $\inf_{t>0} w(t)>0$.
        \end{enumerate}
Let $\phi$ be a Schwartz function supported in a ball centred at the origin. Then $X_w(\R^n)$ is defined to be the set of all locally integrable functions $f$ for which
    \[
        \norm{f}_{X_w(\R^n)}:=\norm{f}_{\BMO(\R^n)}+\sup_{t>0}\frac{\norm{\phi(tD)f}_\infty}{w(t)}<\infty.
    \]
\end{defi}
Motivated by the following proposition, we may think about $X_w(\R^n)$ as intermediate spaces lying in between $L^\infty(\R^n)$ and $\bmo(\R^n)$.

\begin{prop}\label{xw_prop}
    \cite{Paper1}*{Proposition 2.6} Let $w$ and $\phi$ be as in Definition \ref{X_w}. 
    \begin{enumerate}[label=\alph*)]
        \item The definition of the space $X_w(\R^n)$ does not depend on the different choices of function $\phi$, in the sense that different choices induce equivalent norms.
        \item The embeddings $L^\infty(\R^n)\subset X_w(\R^n)\subset \bmo(\R^n)$ hold.
        \item If $w\approx 1$, then $X_w(\R^n)=L^\infty(\R^n)$ with equivalent norms.
        \item \label{d} For $w(t)=1+\log_+ (1/t)$, we have that $X_w(\R^n)=\bmo(\R^n)$ with equivalent norms.
    \end{enumerate}
\end{prop}

We note as well that as a direct consequence of \ref{I},  \ref{II} and \ref{III}, $w$ also satisfies that for all $0<c_1\leq c_2$, there exist $0<d_1\leq d_2$ such that
\begin{equation}\label{equ9}
    c_1\leq \frac{t}{s}\leq c_2\quad \mbox{implies that}\quad  d_1\leq \frac{w(t)}{w(s)}\leq d_2.
\end{equation}

Throughout this article we will refer to the following type of functions as admissible weights, slightly modifying the definition of A. Caetano and S. Moura in \cite{cae-mou}*{Definition 2.1}.

\begin{defi}%
    Let $w:(0,1]\rightarrow (0,\infty)$ be a monotonic function, and extend it to $w:(0,\infty)\to (0,\infty)$ by defining $w(t)=w(1)$ for all $t\geq 1$. We say that $w$ is an \textit{admissible} weight if there exist $c,d>0$ such that
    \begin{equation*}
        cw(2^{-j})\leq  w(2^{-2j}) \leq d w(2^{-j})
    \end{equation*}
     for all $j\in\N$. 
\end{defi}

\begin{ex}
    The prototype of admissible weights we have in mind are those functions of the form
    \[
        w(t):=(1+\log_+(1/t))^{b_1}\brkt{1+\log (1+\log_+(1/t))}^{b_2},
    \]
    with $b_1$, $b_2\in \R$ and $b_1\cdot b_2\geq 0$.
\end{ex}

Admissible weights satisfy similar properties as those weights considered in Definition \ref{X_w}. 
\begin{lem}
    \cite{Paper1}*{Lemma 2.12} Let $w$ be an admissible weight. Then $w$ satisfies \ref{I} and \ref{II} in Definition \ref{X_w} above. Moreover, condition \ref{III} holds for an admissible weight $w$ if, and only if, $w$ is either non-increasing, or satisfies that for all $t>0$, $w(t)\approx 1$.
\end{lem}
We notice that the case where $w\approx 1$ is not of interest, since it is equivalent to the study of the constant weight $w\equiv 1$. We will focus then on non-increasing admissible weights. We also point out that for those weights it is possible to find $b\geq 0$ such that
\begin{equation}\label{equ10}
    \sup_{s>0}\frac{w(ts)}{w(s)}\lesssim (1+\log_+1/t)^b
\end{equation}
for all $t>0$ (see the proof of \cite{Paper1}*{Lemma 2.12}).

\subsection{Function spaces of generalised smoothness}
Let us now recall the definition of function spaces of Besov and Triebel-Lizorkin type with generalised smoothness. We studied mainly the expositions in \cite{cae-mou}, \cite{dominguez2018function} and \cite{mou}.

Let $\varphi_0$ be a positive and radially monotonically decreasing Schwartz function, supported in the ball $\lbrace\abs{\xi}\leq 3/2\rbrace$, which is identically one on $\lbrace\abs{\xi}\leq 1\rbrace$. We define then $\varphi(\xi):=\varphi_0(\xi)-\varphi_0(2\xi)$ and $\varphi_j(\xi):=\varphi(2^{-j}\xi)$ for $\xi\in\R^n$ and all integers $j\geq 1$. We notice that $\varphi_j$ is supported in the annulus $\lbrace 2^{j-1}\leq\abs{\xi}\leq 2^{j+1}\rbrace$ for all $j\geq 1$ and it holds that $\sum_{j=0}^\infty\varphi_j(\xi)=1$ for all $\xi\in\R^n$. Such family $\{\varphi_j\}_{j\geq 0}$ forms a resolution of unity. 

\begin{defi}\label{GSTL}
Let $s\in\R$, $0<p\leq\infty$ and $0<q\leq\infty$. Set $w$ for an admissible weight and let $\lbrace\varphi_j\rbrace_{j=0}^\infty$ be a resolution of unity as above. 
\begin{itemize}
    \item \cites{cae-mou,mou} If $0<p<\infty$, we define the Triebel-Lizorkin space of generalised smoothness, $F_{p,q}^{s,w}(\R^n)$, to be the set of all tempered distributions $f$ for which
\[
    \norm{f}_{F_{p,q}^{s,w}(\R^n)}:=\norm{\left(\sum_{j=0}^\infty 2^{jsq}w(2^{-j})^q\abs{\varphi_j(D)f}^q\right)^{1/q}}_{L^p(\R^n)}<\infty.
\]

\item Let $\mathcal{D}$ be the set of all dyadic cubes in $\R^n$. We define $F_{\infty,q}^{s,w}(\R^n)$ to be the set of all tempered distributions $f$ for which
\begin{align*}
    \norm{f}_{F_{\infty,q}^{s,w}(\R^n)}&:=\norm{\varphi_0(D)f}_\infty\\
    &\quad+\sup_{\substack{Q\in\mathcal{D} \\ \ell(Q)\leq 1}}\left(\frac{1}{\abs{Q}}\int_Q\sum_{\ell=-\log_2 \ell(Q)}^\infty 2^{s\ell q}w(2^{-\ell})^q\abs{\varphi_\ell(D)f(x)}^q \dd x\right)^{1/q}
\end{align*}
is finite.
\item \cites{cae-mou,mou} We define the Besov space of generalised smoothness, $B_{\infty,\infty}^{s,w}(\R^n)=F_{\infty,\infty}^{s,w}(\R^n)$, to be the set of all tempered distributions $f$ for which
\[
    \sup_{j\geq 0}2^{js}w(2^{-j})\norm{\varphi_j(D)f}_\infty<\infty.
\]
\end{itemize}
\end{defi}

As stated in \cite{mou}*{p. 11}, minor modifications of the arguments in \cite{Trie83} for the classical function spaces, the quantities $\norm{\cdot}_{F_{p,q}^{s,w}}$ and $\norm{\cdot}_{B_{\infty,\infty}^{s,w}}$ do not depend on the resolution of unity appearing in the definition, in the sense that different choices of resolution of the unity produce equivalent quasi-norms.

Note that if we consider as admissible weight the constant function $w\equiv 1$, Definition~\ref{GSTL} recovers the classical Triebel-Lizorkin spaces $F_{p,q}^s(\R^n)$ and the Besov space $B_{\infty,\infty}^s(\R^n)$. In addition, it is also known that $F_{p,2}^0(\R^n)$ coincides with $h^p(\R^n)$ when $0<p<\infty$ (see \cite{Trie83}*{Theorem 2.5.8/1}) and that $F_{\infty,\infty}^{s}(\R^n)=B_{\infty,\infty}^{s}(\R^n)$ (see \cite{Runst-Sickel}*{Remark 2.2.3/3}).

\begin{rem}\label{emb_Xw_besov}
    We shall point out that $X_w(\R^n)$ is continuously embedded in the Besov space $B_{\infty,\infty}^0(\R^n)$ for every weight $w$ as in Definition \ref{X_w}. Indeed, fixed such a weight, if we set $\lbrace\varphi_j\rbrace_{j=0}^\infty$ as above, we have that
    \[
        \sup_{j\geq 1}\norm{\varphi_j(D)f}_\infty\lesssim\norm{f}_{\BMO(\R^n)}\leq\norm{f}_{X_w(\R^n)},
    \]
    where the first inequality follows from Proposition \ref{QofBMO} and the second one from the definition of the norm on $X_w(\R^n)$.
    
    The definition of $\norm{\cdot}_{X_w(\R^n)}$ yields
    \[
        \norm{\varphi_0(D)f}_\infty\lesssim w(1)\norm{f}_{X_w(\R^n)}.
    \]
    Hence $\norm{f}_{B_{\infty,\infty}^0(\R^n)}\lesssim\norm{f}_{X_w(\R^n)}$ for every weight $w$ as in Definition \ref{X_w}.
\end{rem}

\subsection{A multiplier theorem} Let us consider a sequence of functions $\lbrace f_\ell\rbrace_{\ell=0}^\infty$ on $\R^n$ and $0<p,q\leq\infty$. We define the $L^p(\ell^q)$-norm of the sequence $\lbrace f_\ell\rbrace_{\ell=0}^\infty$ as
\[
    \norm{\lbrace f_\ell\rbrace_{\ell=0}^\infty}_{L^p(\ell^q)}=\norm{\left(\sum_{\ell=0}^\infty\abs{f_\ell}^q\right)^{1/q}}_{L^p(\R^n)},
\]
with the usual modification if $p=\infty$ or $q=\infty$.

The following result due to B.J. Park (see \cite{Park}*{Lemma 2.7} or \cite{Park2}), which is an extension of \cite{Trie83}*{Theorem 1.6.3}, is an essential tool to prove our main theorem.

\begin{lem}\label{Lplq}
     Let $0<p,q\leq\infty$, $j\geq 0$, $s>\tau_{p,q}+n/2$ and $A>0$. Assume that $\lbrace m_\ell\rbrace_{\ell=0}^\infty$ satisfies
    \[
        \sup_ {\ell\geq 0}\norm{m_ \ell(2^{\ell+j}\cdot)}_{L^2_s(\R^n)}<\infty.
    \]
    \begin{enumerate}[label=(\alph*)]
        \item\label{casopfinita} If $0<p<\infty$ or $p=q=\infty$ then
        \[
            \norm{\lbrace m_\ell(D)f_{\ell+j}\rbrace_{\ell=0}^\infty}_{L^p(\ell^q)}\lesssim\left(\sup_ {\ell\geq 0}\norm{m_ \ell(2^{\ell+j}\cdot)}_{L^2_s(\R^n)}\right)\norm{\lbrace f_\ell\rbrace_{\ell=0}^\infty}_{L^p(\ell^q)}
        \]
        uniformly in $j$, for all $\lbrace f_\ell\rbrace_{\ell=0}^\infty$ such that $\operatorname{supp}\widehat{f_\ell}\subseteq\lbrace\xi\in\R^n:\abs{\xi}\leq A 2^\ell\rbrace$.
        
        \item\label{casopinfty} If $p=\infty$ and $0<q<\infty$ then
        \begin{align*}
            &\sup_{\substack{Q\in\mathcal{D} \\ \ell(Q)\leq 1}}\left(\frac{1}{\abs{Q}}\int_Q \sum_{k=-\log_2 \ell(Q)}^\infty\abs{m_\ell(D)f_{\ell+j}(x)}^q\dd x\right)^{1/q}\\
            &\quad\lesssim\sup_{\ell\geq 1}\norm{m_\ell(2^{\ell+j}\cdot)}_{L^2_s(\R^n)}\sup_{\substack{Q\in\mathcal{D} \\ \ell(Q)\leq 1}}\left(\frac{1}{\abs{Q}}\int_Q \sum_{k=-\log_2 \ell(Q)}^\infty\abs{f_{\ell+j}(x)}^q\dd x\right)^{1/q}
        \end{align*}
        uniformly in $j$, for all $\lbrace f_\ell\rbrace_{\ell=0}^\infty$ such that $\operatorname{supp}\widehat{f_\ell}\subseteq\lbrace\xi\in\R^n:\abs{\xi}\leq A 2^\ell\rbrace$.
    \end{enumerate}
\end{lem}

\subsection{Weighted Hardy inequality} We will make use of the following discrete version of Hardy's weighted inequality that can be found in \cite{Bennet}*{Theorem 9} (see also \cite{KuMaPe}*{Theorem 7}). 
\begin{prop}\label{wei-har}
    Let $\lbrace u_\ell\rbrace_{\ell=0}^\infty$ and $\lbrace v_\ell\rbrace_{\ell=0}^\infty$ be two sequences of positive numbers and $1<q<\infty$. The inequality
    \begin{equation}\label{whardy_ineq}
        \left(\sum_{\ell=0}^\infty\left(\sum_{k=0}^\ell a_k\right)^q u_\ell\right)^{1/q}\leq C\left(\sum_{\ell=0}^\infty a_\ell^q v_\ell\right)^{1/q}
    \end{equation}
    is satisfied for some constant $C>0$ and all positive sequences $\lbrace a_\ell\rbrace_{\ell=0}^\infty$ if, and only if,
    \[
        \mathcal{A}:=\sup_{\ell\geq 0}\left(\sum_{k=\ell}^\infty u_k\right)^{1/q}\left(\sum_{k=0}^\ell v_k^{1-q'}\right)^{1/q'}<\infty,
    \]
    where $1/q+1/q'=1$.
    
    The smallest constant $C$ such that \eqref{whardy_ineq} holds satisfies
    \[
        \mathcal{A}\leq C\leq q^{1/q}(q')^{1/q'}\mathcal{A}.
    \]
\end{prop}

\section{Main results}\label{main_results}

We state the main theorem of this paper.
\begin{thm}\label{main_thm}
    Let us consider the operator $T_\sigma$ with $\sigma$ as in \eqref{main_sym} and let $w$ be a non-increasing admissible weight.
    \begin{enumerate}[label=\Roman*)]
        \item\label{equ31}If $0<p\leq\infty$, $0<q\leq\infty$, $m\in\R$ and $s>\tau_{p,q}$ then we can find $C>0$ such that
        \[
            \norm{T_\sigma(f,g)}_{F_{p,q}^{s,1/w}(\R^n)}\leq C\norm{f}_{F_{p,q}^{s+m}(\R^n)}\norm{g}_{X_w(\R^n)}
        \]
        for all $f,g\in\SW(\R^n)$.
        \item\label{equ25} Let $0<p\leq \infty$, $1<q<\infty$, $m\in\R$, $s>\tau_{p,q}$ and take an admissible weight $v$ such that
        \begin{equation}\label{equ36}
            \sup_{\ell\geq 0}\left(\sum_{k=\ell}^\infty\frac{1}{w(2^{-k})^q}\right)^{1/q}\left(\sum_{k=0}^\ell \frac{1}{v(2^{-k})^{q'}}\right)^{1/q'}<\infty.
        \end{equation}
        Then we can find $C>0$ such that
        \[
            \norm{T_{\sigma}(f,g)}_{F_{p,q}^{s,1/w}(\R^n)}\leq C\norm{f}_{B_{\infty,\infty}^{s+m}(\R^n)}\norm{g}_{F_{p,q}^{0,v}(\R^n)}
        \]
        for all $f,g\in\SW(\R^n)$.
    \end{enumerate}
\end{thm}
\begin{ex}
Let us give some examples of combinations of admissible weights $w$ and $v$ satisfying  \eqref{equ36}.
\begin{enumerate}\label{ex_w}
    \item Let $w(t)=(1+\log_+ 1/t)^{\alpha}$, with $\alpha\geq 1>1/q$, and $v\equiv 1$. The condition in \eqref{equ36} is equivalent to 
    \[
        \sup_{\ell\geq 0}   \brkt{\sum_{k=\ell}^\infty \brkt{1+k}^{-\alpha q}}^{1/q}(1+\ell)^{1/q'}
    \]
    since $w(2^{-k})\approx(1+k)^\alpha$. By using that $\alpha>1/q$ we observe that
    \[
        \brkt{\sum_{k=\ell}^\infty \brkt{1+k}^{-\alpha q}}^{1/q}\leq\left(\int_{\ell-1}^\infty(1+x)^{-\alpha q}\dd x\right)^{1/q}\lesssim \ell^{-\alpha+1/q},
    \]
    from where, since $\alpha\geq 1$ it follows that
    \[
        \sup_{\ell\geq 1}   \brkt{\sum_{k=\ell}^\infty \brkt{1+k}^{-\alpha q}}^{1/q}(1+\ell)^{1/q'}\lesssim\sup_{\ell\geq 1}\ell^{-\alpha+1}<\infty.
    \]
    \item Let $w(t)=(1+\log_+ 1/t)^{\alpha}$ with $\alpha>1/q$ and $v(t)=(1+\log_+ 1/t)^{\beta}$ with $\beta>1/q'$. Then \eqref{equ36} is equivalent to
     \[
        \sup_{\ell\geq 0}   \brkt{\sum_{k=\ell}^\infty \brkt{1+k}^{-\alpha q}}^{1/q}\brkt{\sum_{k=0}^\ell \brkt{1+k}^{-\beta q'}}^{1/q'},
    \]
    which is bounded by
    \[
        \brkt{\sum_{k=0}^\infty \brkt{1+k}^{-\alpha q}}^{1/q}\brkt{\sum_{k=0}^\infty \brkt{1+k}^{-\beta q'}}^{1/q'}.
    \]
    The two previous series are convergent when $\alpha>1/q$ and $\beta>1/q'$.
    \item Let $w(t)=(1+\log_+ 1/t)^\alpha$ and $v(t)=(1+\log_+ 1/t)^{-\beta}$ with $\beta>0$ and $\alpha\geq 1+\beta$. Then \eqref{equ36} is equivalent to
    \[
        \sup_{\ell\geq 0}   \brkt{\sum_{k=\ell}^\infty \brkt{1+k}^{-\alpha q}}^{1/q}\brkt{\sum_{k=0}^\ell \brkt{1+k}^{\beta q'}}^{1/q'}.
    \]
    In that case, we notice that
    \[
        \brkt{\sum_{k=\ell}^\infty \brkt{1+k}^{-\alpha q}}^{1/q}\brkt{\sum_{k=0}^\ell \brkt{1+k}^{\beta q'}}^{1/q'}\lesssim\ell^{-\alpha+1/q}\ell^{\beta+1/q'}=\ell^{-\alpha+\beta+1},
    \]
    where $\sup_{\ell\geq 1}\ell^{-\alpha+\beta+1}<\infty$ when $\alpha\geq 1+\beta$.
\end{enumerate}    
\end{ex}
As a consequence of Theorem \ref{main_thm} we obtain the following boundedness property for pseudodifferential operators with symbol in the bilinear H\"ormander class $BS^m_{1,1}$.
\begin{cor}\label{cor_BS}
    Let $0<p\leq\infty$, $1<q<\infty$, $s>\tau_{p,q}$, $m\in\R$ and  $\sigma\in BS_{1,1}^m$. Set $w$, $v$ admissible weights such that $w$ is decreasing and \eqref{equ36} is satisfied. There exist $C>0$ and a positive integer $N$ such that
    \[
        \norm{T_\sigma(f,g)}_{F_{p,q}^{s,1/w}(\R^n)}\leq C\norm{\sigma}_{BS_{1,1;N}^m}\left(\norm{f}_{F_{p,q}^{s+m}(\R^n)}\norm{g}_{X_w(\R^n)}+\norm{f}_{F_{p,q}^{0,v}(\R^n)}\norm{g}_{B_{\infty,\infty}^{s+m}(\R^n)}\right)
    \]
    for all $f,g\in\SW(\R^n)$.
\end{cor}

\begin{proof}
The decomposition of the symbol follows the strategy used by K. Koezuka and N. Tomita in \cite{koe-tom}*{Theorem 1.1} (see also \cite{ben-tor}). We repeat the details for the sake of completeness.

Let $\sigma\in BS_{1,1}^m$ and let $\lbrace \varphi_j\rbrace_{j=0}^\infty$ be as in Definition \ref{GSTL}. We notice that
    \begin{equation}\label{equ2}
        \sigma(x,\xi,\eta)=\sum_{j=0}^\infty\sum_{k=0}^j\sigma(x,\xi,\eta)\varphi_j(\xi)\varphi_k(\eta)+\sum_{k=1}^\infty\sum_{j=0}^{k-1}\sigma(x,\xi,\eta)\varphi_j(\xi)\varphi_k(\eta).
    \end{equation}
    Let us define $\sigma^0$ as the former sum and $\sigma^1$ as the later one. We write
    \begin{equation}\label{equ1}
        \sigma^0(x,\xi,\eta)=\sum_{j=0}^\infty\sum_{k=0}^j\sigma(x,\xi,\eta)\varphi_j(\xi)\varphi_k(\eta)=\sum_{j=0}^\infty\sigma_j(x,\xi,\eta)
    \end{equation}
    with $\sigma_j(x,\xi,\eta)=\sigma(x,\xi,\eta)\varphi_j(\xi)\varphi_0(2^{-j}\eta)$, $j\geq 0$.
    
    Next let us take $\chi_0,\chi\in\SW(\R^n)$ with $\chi_0$ supported in the ball $\lbrace\abs{\xi}\leq 3\rbrace$ and $\chi$ supported in the ring $\lbrace 1/3\leq\abs{\xi}\leq 3\rbrace$, with $\chi_0$ and $\chi$ being identically one on $\lbrace\abs{\xi}\leq 2\rbrace$ and $\lbrace 1/2\leq\abs{\xi}\leq 2\rbrace$ respectively. Define $\chi_j(\xi):=\chi(2^{-j}\xi)$ for $j\geq 1$, in such a way that the functions $\chi_j$ are identically one in the support of $\varphi_j$ for all $j\geq 0$.
    
    Then we consider the Fourier series expansion
    \[
        \sigma_j(x,\xi,\eta)=\sum_{k,\ell\in\Z^n}c_{j,k,\ell}(x)e^{ik(2^{-j}\xi)}\varphi_j(\xi)e^{i\ell(2^{-j}\eta)}\varphi_0(2^{-j}\eta)
    \]
    where
    \[
        c_{j,k,\ell}=\frac{1}{(2\pi)^n}\iint\sigma(x,2^j\xi,2^j\eta)\chi(\xi)\chi_0(\eta)e^{-i(k\xi+\ell\eta)}\dd\xi\dd\eta,
    \]
    with $\chi_0(\xi)$ instead of $\chi(\xi)$ when $j=0$.
    
    Using an integration by parts argument it follows that for any positive integers $a,b,N$ there is $N'\in\N$ such that
    \begin{equation}\label{eqcoef}
        \sup_{j\geq 0,k,\ell\in\Z^n}2^{-j(m+\abs{\alpha})}(1+\abs{k})^a(1+\abs{\ell})^b\norm{\partial^\alpha c_{j,k,\ell}}_\infty\lesssim\norm{\sigma}_{BS^m_{1,1;N'}}
    \end{equation}
    for all $\abs{\alpha}\leq N$.
    
    For any positive integers $a',a'',b',b''$ we can write $\sigma_j$ as
    \[
        \sum_{k,\ell\in\Z^n}(1+\abs{k})^{-a'}(1+\abs{\ell})^{-b'}\m_j^{(k,\ell)}(x)\psi_j^{(k)}(\xi)\phi_j^{(\ell)}(\eta)
    \]
    where
    \begin{align*}
        &\m_j^{(k,\ell)}(x)=(1+\abs{k})^{a'+a''}(1+\abs{\ell})^{b'+b''}c_{j,k,\ell}(x),\\
        &\psi_j^{(k)}(\xi)=(1+\abs{k})^{-a''}e^{ik(2^{-j}\xi)}\varphi_j(\xi),\\
        &\phi_j^{(\ell)}(\eta)=(1+\abs{\ell})^{-b''}e^{i\ell(2^{-j}\eta)}\varphi_0(2^{-j}\eta).
    \end{align*}
    Therefore, the symbol $\sigma^0$ can be then written as
    \[
       \sum_{k,\ell\in\Z^n}(1+\abs{k})^{-a'}(1+\abs{\ell})^{-b'}\left(\sum_{j=0}^\infty \m_j^{(k,\ell)}(x)\psi_j^{(k)}(\xi)\phi_j^{(\ell)}(\eta)\right).
    \]
    We observe that by using \eqref{eqcoef} it is possible to find, for a given $a,b,N\in\N$, a positive integer $N'$ such that
    \[
        \abs{\partial^\alpha\m_j^{(k,\ell)}(x)}\lesssim 2^{j(m+\abs{\alpha})}(1+\abs{k})^{a'+a''-a}(1+\abs{\ell})^{b'+b''-b}\norm{\sigma}_{BS^m_{1,1;N'}}
    \]
    for all $x\in\R^n$ and $\abs{\alpha}\leq N$. In addition, we can also find $C_N>0$ such that
    \[
        \abs{\partial^\alpha\psi_j^{(k)}(\xi)}\leq C_N 2^{-j\abs{\alpha}}(1+\abs{k})^{-a''+N}
    \]
    and
    \[
        \abs{\partial^\alpha\phi_j^{(\ell)}(\eta)}\leq C_N 2^{-j\abs{\alpha}}(1+\abs{\ell})^{-b''+N}
    \]
    for all $\xi,\eta\in\R^n$ and $\abs{\alpha}\leq N$.  Hence by choosing $a',a'',b',b''$ large enough we can reduce ourselves to establish the desired boudnedness properties for the symbol in \eqref{main_sym}. Then we can apply Theorem \ref{main_thm} \ref{equ31} to obtain that
    \begin{equation}\label{sigma0}
            \norm{T_{\sigma^0}(f,g)}_{F_{p,q}^{s,1/w}(\R^n)}\lesssim\norm{f}_{F_{p,q}^{s+m}(\R^n)}\norm{g}_{X_w(\R^n)}.
    \end{equation}
    An analogous argument can be repeated to reduce the study of the symbol $\sigma^1$ to a symbol as in \eqref{main_sym}, so Theorem \ref{main_thm} \ref{equ25} yields
    \begin{equation}\label{sigma1}
            \norm{T_{\sigma^1}(g,f)}_{F_{p,q}^{s,1/w}(\R^n)}\leq C\norm{g}_{B_{\infty,\infty}^{s+m}(\R^n)}\norm{f}_{F_{p,q}^{0,v}(\R^n)}.
    \end{equation}
    The proof finishes by combining the estimates \eqref{sigma0} and \eqref{sigma1}.
\end{proof}
\begin{cor}\label{cor_BS_dia}
    Let $0<p\leq\infty$, $0<q\leq\infty$, $s>\tau_{p,q}$, $m\in\R$ and  $\sigma\in BS_{1,1}^m$. There exist $C>0$ and a positive integer $N$ such that
    \begin{align*}
        \norm{T_\sigma(f,g)}_{F_{p,q}^{s,1/(1+\log_+1/t)}(\R^n)}&\leq C\norm{\sigma}_{BS_{1,1;N}^m}\\
        &\quad\times\left(\norm{f}_{F_{p,q}^{s+m}(\R^n)}\norm{g}_{\bmo(\R^n)}+\norm{f}_{\bmo(\R^n)}\norm{g}_{F_{p,q}^{s+m}(\R^n)}\right)
    \end{align*}
    for all $f,g\in\SW(\R^n)$.
\end{cor}
\begin{proof}
    The proof is analogous to the one of Corollary \ref{cor_BS}. However, this time we shall apply Theorem \ref{main_thm} \ref{equ31} to both $T_{\sigma^0}$ and $T_{\sigma^1}$.
    
    In addition, we should notice that when $w(t)=1+\log_+(1/t)$ then $X_w(\R^n)$ becomes $\bmo(\R^n)$, as stated in Proposition \ref{xw_prop} \ref{d}. More precisely, we have that $\norm{f}_{X_w(\R^n)}\approx\norm{f}_{\bmo(\R^n)}$ and $\norm{g}_{X_w(\R^n)}\approx\norm{g}_{\bmo(\R^n)}$.
\end{proof}
Finally, we can also obtain the following endpoint Kato-Ponce inequality from Corollary \ref{cor_BS}, involving $\bmo(\R^n)$.
\begin{cor}\label{cor_kp}
    Let $0<p\leq\infty$, $1<q<\infty$ and $s>\tau_{p,q}$. Then
    \[
         \norm{J^s(fg)}_{F_{p,q}^{0,1/(1+\log_+ (1/t))}(\R^n)}\lesssim\norm{J^sf}_{F_{p,q}^0(\R^n)}\norm{g}_{\bmo(\R^n)}+\norm{f}_{F_{p,q}^0(\R^n)}\norm{J^sg}_{\bmo(\R^n)}
    \]
    for all $f,g\in\SW(\R^n)$.
\end{cor}
\begin{proof}
    Applying Corollary \ref{cor_BS} with $\sigma\equiv 1$, $m=0$, $w(t)=1+\log_+(1/t)$ and $v\equiv 1$ (see Example \ref{ex_w} above), jointly with Proposition \ref{xw_prop} \ref{d} we obtain that
     \[
        \norm{fg}_{F_{p,q}^{s,1/(1+\log_+ (1/t))}(\R^n)}\lesssim \norm{f}_{F_{p,q}^s(\R^n)}\norm{g}_{\bmo(\R^n)}+\norm{f}_{F_{p,q}^0(\R^n)}\norm{g}_{B_{\infty,\infty}^s(\R^n)}.
    \]
   By using the lifting property for Besov and Triebel-Lizorkin spaces with generalised smoothness shown in \cite{mou}*{Proposition 1.8} we see that
    \[
        \norm{J^s(fg)}_{F_{p,q}^{0,1/(1+\log_+ (1/t))}(\R^n)}\lesssim \norm{J^sf}_{F_{p,q}^0(\R^n)}\norm{g}_{\bmo(\R^n)}+\norm{f}_{F_{p,q}^0(\R^n)}\norm{J^sg}_{B_{\infty,\infty}^0(\R^n)}.
    \]
    To finish the proof we notice that Remark \ref{emb_Xw_besov} jointly with Proposition \ref{xw_prop} \ref{d} yield
    \[
        \norm{J^sg}_{B_{\infty,\infty}^0(\R^n)}\lesssim\norm{J^sg}_{\bmo(\R^n)}.
    \]
\end{proof}
\begin{rem}
    It is known (see \cite{Runst-Sickel}*{Proposition 2.1.2}) that $\norm{\cdot}_{\bmo(\R^n)}\approx\norm{\cdot}_{F_{\infty,2}^0(\R^n)}$. Then Corollary \ref{cor_kp} with $p=\infty$ and $q=2$ yields, for $s>0$,
    \[
         \norm{J^s(fg)}_{F_{\infty,2}^{0,1/(1+\log_+ (1/t))}(\R^n)}\lesssim\norm{J^sf}_{\bmo(\R^n)}\norm{g}_{\bmo(\R^n)}+\norm{f}_{\bmo(\R^n)}\norm{J^sg}_{\bmo(\R^n)}.
    \]
     This inequality complements  that obtained in \cite{rod-sta}*{Theorem 7.5}, with a weaker restriction on the regularity index.
\end{rem}
\section{Proof of Theorem \ref{main_thm}}\label{Proof_main_thm}
\subsection{Proof of part \ref{equ31}}. We follow partially the strategy used by K. Koezuka and N. Tomita to prove \cite{koe-tom}*{Theorem 1.1}.

Let us consider the collection $\lbrace\varphi_j\rbrace_{j=0}^\infty$ from Definition \ref{GSTL}. We notice that
\[
    \varphi_0(2^{-j}\xi)+\sum_{k=1}^\infty\varphi(2^{-j-k}\xi)=1
\]
for all $\xi\in\R^n$ and $j\geq 1$, so that we can decompose the functions $\m_j$ defining $\sigma$ as $\m_j=\m_{j,0}+\sum_{k=1}^\infty \m_{j,k}$ with
\[
    \m_{j,0}=\varphi_0(2^{-j}D)\m_j \quad \mathrm{and} \quad \m_{j,k}=\varphi(2^{-j-k}D)\m_j,\quad k\geq 1.
\]
Set $\Phi:=\mathcal{F}^{-1}\varphi$ and pick $N\in\N$. Using the fact that $\partial^\alpha\varphi(0)=0$ for all $\alpha\in\N^n$ jointly with Taylor's formula we deduce that
\begin{align*}
    \m_{j,k}(x)&=2^{(j+k)n}\int_{\R^n}\Phi(2^{j+k}(x-y))\left(\m_j(y)-\sum_{\abs{\alpha}<N}\frac{\partial^\alpha\m_j(x)}{\alpha !}(y-x)^\alpha\right)\dd y\\
    &=2^{(j+k)n}\int_{\R^n}\Phi(2^{j+k}(x-y))\\
    &\quad\times\left(N\sum_{\abs{\alpha}=N}\frac{(y-x)^\alpha}{\alpha !}\int_0^1(1-t)^{N-1}(\partial^\alpha \m_j)(x+t(y-x))\dd t\right)\dd y.
\end{align*}
Then we use \eqref{equ4} to obtain that
\begin{equation}\label{equ8}
    \norm{\m_{j,k}}_\infty\lesssim 2^{jm-kN}
\end{equation}
for all $j\geq 0$ and $k\geq 1$. The same inequality holds in the case $k=0$, which does not require Taylor's formula to be deduced.
    
Using the previous decomposition of $\m_j$ we can write
\[
    T_\sigma(f,g)(x)=\sum_{j=0}^\infty\sum_{k=0}^\infty \m_{j,k}(x)\left(\psi_j(D)f\right)(x)\left(\phi_j(D)g\right)(x),
\]
where the frequency support of each function in the sum satisfies
\begin{equation}\label{equ5}
    \operatorname{supp}\mathcal{F}[\m_{j,k}\left(\psi_j(D)f\right)\left(\phi_j(D)g\right)]\subseteq\lbrace\abs{\xi}\lesssim 2^{j+k}\rbrace,\quad j,k\geq 0.
\end{equation}
Next we define $\m_{j,k}:=0$ for all $j<0$, $k\geq 0$, and $\psi_j=\phi_j:=0$ for all $j<0$. Combining the support condition on $\varphi_\ell$ with \eqref{equ5} we get that
\[
    \varphi_\ell(D)T_\sigma(f,g)=\sum_{k=0}^\infty\sum_{j+k+L\geq\ell}^\infty\varphi_\ell(D)[\m_{j,k}\left(\psi_j(D)f\right)\left(\phi_j(D)g\right)],
\]
where $L$ is a fixed integer determined by the implicit constants for the sizes of the supports of $\varphi_\ell$, $\widehat{\m_{j,k}}$, $\psi_j$ and $\phi_j$. Hence a change of variables yields
\begin{equation}\label{DecomLPpiece}
    \varphi_\ell(D)T_\sigma(f,g)=\sum_{k=0}^\infty\sum_{j=0}^\infty\varphi_\ell(D)\left[\m_{j_{(k,\ell)},k}\left(\psi_{j_{(k,\ell)}}(D)f\right)\left(\phi_{j_{(k,\ell)}}(D)g\right)\right]
\end{equation}
with $j_{(k,\ell)}=j-k+\ell-L$.

\subsubsection{Case $0<p<\infty$ or $p=q=\infty$} If we define $r=\min\lbrace p,q,1\rbrace$ then it holds that
$$
    \norm{\lbrace f_j\rbrace_j+\lbrace g_j\rbrace_j}_{L^p(\ell^q)}^r\leq \norm{\lbrace f_j\rbrace_j}_{L^p(\ell^q)}^r+\norm{\lbrace g_j\rbrace_j}_{L^p(\ell^q)}^r.
$$
Using this triangular inequality jointly with \eqref{DecomLPpiece}, we see that the $r$-th power of $\norm{T_\sigma(f,g)}_{F_{p,q}^{s,1/w}(\R^n)}$ can be estimated by
\begin{equation}\label{equ15}
    \sum_{j,k=0}^\infty\left\lbrace\int_ {\R^n}\left(\sum_ {\ell=0}^\infty\abs{\frac{2^{\ell s}}{w(2^{-\ell})}\varphi_\ell(D)\left[\m_{j_{(k,\ell)},k}\psi_{j_{(k,\ell)}}(D)f\phi_{j_{(k,\ell)}}(D)g\right](x)}^q\right)^{p/q}\dd x\right\rbrace^{r/p}.
\end{equation}
Now let us pick $\tilde{s}\in\R$ such that $s>\tilde{s}-n/2>\tau_{p,q}$, and since
\begin{equation}\label{suppcond}
    \operatorname{supp}\mathcal{F}\left[\m_{j_{(k,\ell)},k}\left(\psi_{j_{(k,\ell)}}(D)f\right)\left(\phi_{j_{(k,\ell)}}(D)g\right)\right]\subseteq\lbrace\abs{\xi}\lesssim 2^{j+\ell}\rbrace
\end{equation}
for all $j\in\Z$ and $k\geq 0$, we can use Lemma \ref{Lplq}\ref{casopfinita} to obtain that
\begin{align}\label{equ12}
        &\left\lbrace\int_{\R^n}\left(\sum_{\ell=0}^\infty\abs{\frac{2^{\ell s}}{w(2^{-\ell})}\varphi_\ell(D)\left[\m_{j_{(k,\ell)},k}\psi_{j_{(k,\ell)}}(D)f\phi_{j_{(k,\ell)}}(D)g\right](x)}^q\right)^{p/q}\dd x\right\rbrace^{1/p}\nonumber\\
        &\lesssim\left(\sup_{\ell\geq 0}\norm{\varphi_\ell(2^{j+\ell}\cdot)}_{L^2_{\tilde{s}}(\R^n)}\right)\nonumber\\
        &\quad\times\left\lbrace\int_{\R^n}\left(\sum_{\ell=0}^\infty\abs{\frac{2^{(\ell-j) s}}{w(2^{-\ell+j})}\left[\m_{\ell-k-L,k}\psi_{\ell-j-L}(D)f\phi_{\ell-k-L}(D)g\right](x)}^q\right)^{p/q}\dd x\right\rbrace^{1/p}
    \end{align}
Next we use that $\operatorname{supp}\varphi_\ell(2^{j+\ell}\cdot)\subseteq\lbrace\abs{\xi}\lesssim 2^{-j}\rbrace$ and $\norm{\partial^\alpha\varphi_\ell(2^{j+\ell}\cdot)}_\infty\lesssim 2^{j\abs{\alpha}}$ to get that
\begin{equation}\label{equ13}
    \sup_{\ell\geq 0}\norm{\varphi_\ell(2^{j+\ell}\cdot)}_{L^2_{\tilde{s}}(\R^n)}\leq\sup_{\ell\geq 0}\norm{\varphi_\ell(2^{j+\ell}\cdot)}_{L^2(\R^n)}^{1-\theta}\norm{\varphi_\ell(2^{j+\ell}\cdot)}_{L^2_{[\tilde{s}]+1}(\R^n)}^\theta\lesssim 2^{j(\tilde{s}-n/2)},
\end{equation}
where $[\tilde{s}]$ denotes the integer part of $\tilde{s}$ and $\theta=(\tilde{s}+1)/\tilde{s}$.

Now let us focus on the second factor in the right hand side of \eqref{equ12}. To start with we see that Definition \ref{X_w} yields
\begin{equation}\label{equ7}
    \abs{\phi_\ell(D)g(x)}\lesssim w(2^{-\ell})\norm{g}_{X_w(\R^n)}
\end{equation}
for all $\ell\geq 0$ and $x\in\R^n$. Then we change variables and apply \eqref{equ7} jointly with \eqref{equ8} to deduce that
\begin{equation}\label{equ32}
\begin{split}
    &\left\lbrace\int_{\R^n}\left(\sum_{\ell=0}^\infty\abs{\frac{2^{(\ell-j) s}}{w(2^{-\ell+j})}\left[\m_{\ell-k-L,k}\psi_{\ell-j-L}(D)f\phi_{\ell-k-L}(D)g\right](x)}^q\right)^{p/q}\dd x\right\rbrace^{1/p}\\
    &=\left\lbrace\int_{\R^n}\left(\sum_{\ell=-k-L}^\infty\abs{\frac{2^{(-j+\ell+k+L) s}}{w(2^{j-\ell-k-L})}\left[\m_{\ell,k}\psi_{\ell}(D)f\phi_{\ell}(D)g\right](x)}^q\right)^{p/q}\dd x\right\rbrace^{1/p}\\
    &\lesssim 2^{(-j+k)s}\left(\int_ {\R^n}\left(\sum_ {\ell\geq 0}\abs{2^{\ell s}\frac{w(2^{-\ell})}{w(2^{j-k-\ell-L})}\left[\m_{\ell,k}\psi_\ell(D)f\right](x)}^q\right)^{p/q}\dd x\right)^{1/p}\norm{g}_{X_w(\R^n)}\\
    &\lesssim 2^{-js-k(N-s)}\left(\int_ {\R^n}\left(\sum_ {\ell\geq 0}\abs{2^{\ell (s+m
    )}\frac{w(2^{-\ell})}{w(2^{j-k-\ell-L})}[\psi_\ell(D)f](x)}^q\right)^{p/q}\dd x\right)^{1/p}\norm{g}_{X_w(\R^n)}.
\end{split}
\end{equation}
    
In addition, using \eqref{equ9}, the fact that $w$ is decreasing and \eqref{equ10} yields
\begin{equation}\label{equ26}
    \frac{w(2^{-\ell})}{w(2^{j-k-\ell-L})}\lesssim\frac{w(2^{-\ell})}{w(2^{j-k-\ell})}\leq\frac{w(2^{-\ell})}{w(2^{j-\ell})}\lesssim (1+\log_+2^j)^b\lesssim (1+j)^b
\end{equation}
for some $b\geq 0$, from where
\begin{align}\label{equ33}
    &2^{-js-k(N-s)}\left(\int_ {\R^n}\left(\sum_ {\ell\geq 0}\abs{2^{\ell(s+m)}\frac{w(2^{-\ell})}{w(2^{j-k-\ell-L})}[\psi_\ell(D)f](x)}^q\right)^{p/q}\dd x\right)^{1/p}\norm{g}_{X_w(\R^n)}\nonumber\\
    &\lesssim 2^{-js-k(N-s)}(1+j)^b\left(\int_ {\R^n}\left(\sum_ {\ell\geq 0}\abs{2^{\ell(s+m)}[\psi_\ell(D)f](x)}^q\right)^{p/q}\dd x\right)^{1/p}\norm{g}_{X_w(\R^n)}\\
    &\lesssim 2^{-js-k(N-s)}(1+j)^b\norm{f}_{F_{p,q}^{s+m}(\R^n)}\norm{g}_{X_w(\R^n)}\nonumber.
\end{align}
Then we conclude from \eqref{equ15}, \eqref{equ12}, \eqref{equ13}, \eqref{equ32}, and \eqref{equ33} that
\begin{align}\label{equ23}
    \Vert T_\sigma&(f,g) \Vert_{F_ {p,q}^{s,1/w}(\R^n)}\lesssim\left(\sum_{j,k=0}^\infty 2^{j(\tilde{s}-n/2)r}2^{-jsr-k(N-s)r}(1+j)^{br}\norm{f}_{F_{p,q}^{s+m}(\R^n)}^r\norm{g}_{X_w(\R^n)}^r\right)^{1/r}\nonumber\\
    &=\left(\sum_{j=0}^\infty 2^{-j(s-\tilde{s}+n/2)r}(1+j)^{br}\right)^{1/r}\left(\sum_{k=0}^\infty 2^{-k(N-s)r}\right)^{1/r}\norm{f}_{F_{p,q}^{s+m}(\R^n)}\norm{g}_{X_w(\R^n)}\nonumber\\
    &\lesssim\norm{f}_{F_{p,q}^{s+m}(\R^n)}\norm{g}_{X_w(\R^n)},
\end{align}
where the series in $j$ is finite by the ratio test since $(s-\tilde{s}+n/2)r>0$, while the series in $k$ is finite if we choose $N$ large enough.

\subsubsection{Case $p=\infty$ and $0<q<\infty$} \label{finite_case}

We start by setting $r=\min\lbrace 1,q\rbrace$ and taking a cube $Q\in\mathcal{D}$ such that $\ell(Q)\leq 1$. Decomposing $\varphi_\ell(D)T_\sigma(f,g)$ as in \eqref{DecomLPpiece} and using the triangular inequality we have that
\begin{equation}\label{eqi1}
\begin{split}
        &\left(\frac{1}{\abs{Q}}\int_Q\sum_{\ell=-\log_2 \ell(Q)}^\infty\frac{2^{s\ell q}}{w(2^{-\ell})^q}\abs{\varphi_\ell(D)T_\sigma(f,g)(x)}^q \dd x\right)^{1/q}\\
        &\leq\left(\sum_{j,k=0}^\infty\left(\frac{1}{\abs{Q}}\int_Q\sum_{\ell=-\log_2 \ell(Q)}^\infty\frac{2^{s\ell q}}{w(2^{-\ell})^q}\right.\right.\\
        &\left.\left.\quad\times\abs{\varphi_\ell(D)\left[\m_{j_{(k,\ell)},k}\psi_{j_{(k,\ell)}}(D)f\phi_{j_{(k,\ell)}}(D)g\right]}^q \dd x\right)^{ r/q}\right)^{1/r}.
\end{split}
    \end{equation}
Next we pick $\tilde{s}\in\R$ satisfying $s>\tilde{s}-n/2>\tau_{p,q}$ and due to the support condition in \eqref{suppcond} we can apply Lemma \ref{Lplq}\ref{casopinfty} to obtain that
\begin{align}\label{eqi2}
   &\left(\frac{1}{\abs{Q}}\int_Q\sum_{\ell=-\log_2 \ell(Q)}^\infty\frac{2^{s\ell q}}{w(2^{-\ell})^q}\abs{\varphi_\ell(D)\left[\m_{j_{(k,\ell)},k}\psi_{j_{(k,\ell)}}(D)f\phi_{j_{(k,\ell)}}(D)g\right]}^q \dd x\right)^{ 1/q}\nonumber\\
   &\quad\lesssim\left(\sup_{\ell\geq 1}\norm{\varphi_\ell(2^{\ell+j}\cdot)}_{L^2_{\tilde{s}}(\R^n)}\right)\nonumber\\
   &\quad\quad\times\sup_{\substack{Q\in\mathcal{D} \\ \ell(Q)\leq 1}}\left(\frac{1}{\abs{Q}}\int_Q\sum_{\ell=-\log_2 \ell(Q)}^\infty\frac{2^{s\ell q}}{w(2^{-\ell})^q}\abs{\m_{j_{(k,\ell)},k}\psi_{j_{(k,\ell)}}(D)f\phi_{j_{(k,\ell)}}(D)g}^q \dd x\right)^{1/q}.
\end{align}
Next we change variables and apply the inequalities in \eqref{equ7}, \eqref{equ8} and \eqref{equ26} to deduce that
\begin{align}\label{eqi3}
    &\left(\frac{1}{\abs{Q}}\int_Q\sum_{\ell=-\log_2 \ell(Q)}^\infty\frac{2^{s\ell q}}{w(2^{-\ell})^q}\abs{\m_{j_{(k,\ell)},k}\psi_{j_{(k,\ell)}}(D)f\phi_{j_{(k,\ell)}}(D)g}^q \dd x\right)^{1/q}\nonumber\\
    &=\left(\frac{1}{\abs{Q}}\int_Q\sum_{\ell=j-k-\log_2 \ell(Q)-L}^\infty\frac{2^{s(\ell-j+k+L) q}}{w(2^{-\ell-k-L+j})^q}\abs{\m_{\ell,k}\psi_{\ell}(D)f\phi_{\ell}(D)g}^q \dd x\right)^{1/q}\nonumber\\
    &\lesssim 2^{(-j+k)s}\left(\frac{1}{\abs{Q}}\int_Q\sum_{\ell=j-k-\log_2 \ell(Q)-L}^\infty\frac{2^{\ell sq}w(2^{-\ell})^q}{w(2^{-\ell-k-L+j})^q}\abs{\m_{\ell,k}\psi_{\ell}(D)f}^q \dd x\right)^{1/q}\norm{g}_{X_w(\R^n)}\nonumber\\
    &\lesssim 2^{-js-k(N-s)}(1+j)^b\left(\frac{1}{\abs{Q}}\int_Q\sum_{\ell=(j-k-L-\log_2 \ell(Q))_+}^\infty2^{\ell( s+m)q}\abs{\psi_{\ell}(D)f}^q \dd x\right)^{1/q}\norm{g}_{X_w(\R^n)},
\end{align}
where $t_+:=\max\{t,0\}$. 

Now we shall estimate the term
\begin{equation}\label{serrahima}
    \left(\frac{1}{\abs{Q}}\int_Q\sum_{\ell=(j-k-L-\log_2 \ell(Q))_+}^\infty2^{\ell( s+m)q}\abs{\psi_{\ell}(D)f}^q \dd x\right)^{1/q}.
\end{equation}
At this point we should make a distinction, and treat the following three cases separately:
\begin{enumerate}[label=\alph*)]
    \item\label{case1indices} $j-k-L\geq 0$,
    \item\label{case2indices} $\log_2\ell(Q)\leq j-k-L<0$,
    \item\label{case3indices} $j-k-L<\log_2\ell(Q)$.
\end{enumerate}
For the case \ref{case1indices} we obtain that
\begin{equation}\label{eqi4}
    \begin{split}
    &\left(\frac{1}{\abs{Q}}\int_Q\sum_{\ell=j-k-L-\log_2 \ell(Q)}^\infty 2^{\ell( s+m)q}\abs{\psi_{\ell}(D)f}^q \dd x\right)^{1/q}\\
    &\quad\leq \left(\frac{1}{\abs{Q}}\int_Q\sum_{\ell=-\log_2 \ell(Q)}^\infty 2^{\ell( s+m)q}\abs{\psi_{\ell}(D)f}^q \dd x\right)^{1/q}\lesssim \norm{f}_{F_{\infty,q}^{s+m}}.
    \end{split}
\end{equation}
Next we focus on case \ref{case2indices}. Let us pick a second cube $\tilde{Q}\in\mathcal{D}$ which contains $Q$ and such that $\ell(\tilde{Q})=2^{-j+k+L}\ell(Q)$. We notice that $\ell(\tilde{Q})\leq 1$ under the current assumptions. Hence we obtain that
\[
    \sum_{\ell=j-k-L-\log_2 \ell(Q)}^\infty 2^{\ell( s+m)q}\abs{\psi_{\ell}(D)f}^q= \sum_{\ell=-\log_2\ell(\tilde{Q})}^\infty 2^{\ell( s+m)q}\abs{\psi_{\ell}(D)f}^q.
\]
Averaging over $Q$ yields
\begin{align*}
     &\frac{1}{\abs{Q}}\int_Q \sum_{\ell=j-k-L-\log_2 \ell(Q)}^\infty 2^{\ell( s+m)q}\abs{\psi_{\ell}(D)f}^q\\ &\quad\leq\frac{2^{(-j+k+L)n}}{\vert\tilde{Q}\vert}\int_{\tilde{Q}}\sum_{-\log_2\ell(\tilde{Q})}^\infty 2^{\ell( s+m)q}\abs{\psi_{\ell}(D)f}^q,
\end{align*}
from where we deduce that
\begin{equation}\label{eqi6}
    \begin{split}
        &\left(\frac{1}{\abs{Q}}\int_Q\sum_{\ell=j-k-L-\log_2 \ell(Q)}^\infty 2^{\ell( s+m)q}\abs{\psi_{\ell}(D)f}^q \dd x\right)^{1/q}\\
        &\leq 2^{(-j+k+L)n/q}\left(\frac{1}{\vert\tilde{Q}\vert}\int_{\tilde{Q}}\sum_{-\log_2\ell(\tilde{Q})}^\infty 2^{\ell( s+m)q}\abs{\psi_{\ell}(D)f}^q \dd x\right)^{1/q}\\ &\lesssim 2^{(-j+k+L)n/q}\norm{f}_{F_{\infty,q}^{s+m}}.
    \end{split}
\end{equation}
Finally, let us study case \ref{case3indices}. This time we will pick a cube $\tilde{Q}\in\mathcal{D}$ which contains $Q$ and such that $\ell(\tilde{Q})=1$. Since $\ell(Q)>2^{j-k-L}$ it follows that 
\begin{equation}\label{eqi9}
    \begin{split}
    &\left(\frac{1}{\abs{Q}}\int_Q\sum_{\ell=0}^\infty 2^{\ell(  s+m)q}\abs{\psi_{\ell}(D)f}^q \dd x\right)^{1/q}\\
    &=\left(\frac{1}{\abs{Q}}\int_Q\sum_{\ell=-\log_2 \ell(\tilde{Q})}^\infty 2^{\ell( s+m)q}\abs{\psi_{\ell}(D)f}^q \dd x\right)^{1/q}\\
    &\leq 2^{(-j+k+L)n/q}\left(\int_{\tilde{Q}}\sum_{\ell=-\log_2 \ell(\tilde{Q})}^\infty2^{\ell( s+m)q}\abs{\psi_{\ell}(D)f}^q \dd x\right)^{1/q}\\
    &\lesssim 2^{(-j+k+L)n/q}\norm{f}_{F_{\infty,q}^{s+m}}.
    \end{split}
\end{equation}
Therefore, putting together the inequalities \eqref{eqi1}, \eqref{eqi2}, \eqref{equ13}, \eqref{eqi3}, \eqref{eqi4}, \eqref{eqi6} and \eqref{eqi9} we obtain that
\begin{align*}
    &\sup_{\substack{Q\in\mathcal{D} \\ \ell(Q)\leq 1}}\left(\frac{1}{\abs{Q}}\int_Q\sum_{\ell=-\log_2 \ell(Q)}^\infty\frac{2^{s\ell q}}{w(2^{-\ell})^q}\abs{\varphi_\ell(D)T_\sigma(f,g)(x)}^q \dd x\right)^{1/q}\\
    &\quad\lesssim c(L,r,s,N)\norm{f}_{F_{\infty,q}^{s+m}(\R^n)}\norm{g}_{X_w(\R^n)}\\
\end{align*}
where $c(L,r,s,N)^r$ is given by
\begin{equation}\label{tort}
    \sum_{j-k< L} 2^{-j(s-\tilde{s}+n/2+n/q)r}(1+j)^{br}2^{-k(N-s)r}+\sum_{j-k\geq L} 2^{-j(s-\tilde{s}+n/2)r}(1+j)^{br}2^{-k(N-s)r},
\end{equation}
which converges by taking $N$ large enough.

We have seen so far that the inequality
\begin{align}\label{eqi7}
    &\sup_{\substack{Q\in\mathcal{D} \\ \ell(Q)\leq 1}}\left(\frac{1}{\abs{Q}}\int_Q\sum_{\ell=-\log_2 \ell(Q)}^\infty\frac{2^{s\ell q}}{w(2^{-\ell})^q}\abs{\varphi_\ell(D)T_\sigma(f,g)(x)}^q \dd x\right)^{1/q}\nonumber\\
    &\quad\lesssim\norm{f}_{F_{\infty,q}^{s+m}(\R^n)}\norm{g}_{X_w(\R^n)}
\end{align}
is satisfied. On the other hand, we can apply Theorem \ref{main_thm} for the case $p=q=\infty$ jointly with the embedding $F_{\infty,q}^{s+m}\hookrightarrow F_{\infty,\infty}^{s+m}$ (see \cite{Runst-Sickel}*{Remark 2.2.3/3}) to get that
\begin{equation}\label{eqi8}
    \norm{\varphi_0(D)T_\sigma(f,g)}_\infty\leq\norm{T_\sigma(f,g)}_{F_{\infty,\infty}^{s,1/w}(\R^n)}\lesssim\norm{f}_{F_{\infty,q}^{s+m}(\R^n)}\norm{g}_{X_w(\R^n)}.
\end{equation}
Then the statement follows by combining \eqref{eqi7} and \eqref{eqi8}.

\subsection{Proof of part \ref{equ25}} As in the previous part, we will study separately the cases where $p$ is finite and where $p=\infty$. 

\subsubsection{Case $0<p<\infty$} We repeat the argument used to obtain equations \eqref{equ15}, \eqref{equ12} and \eqref{equ13} so that the $r$-th power of $\norm{T_\sigma (f,g)}_{F_{p,q}^{s,1/w}(\R^n)}$ is controlled by
\begin{equation}\label{aa}
    \sum_{j,k=0}^\infty2^{rj(\tilde{s}-n/2)}\left\lbrace\int_ {\R^n}\left(\sum_ {\ell=0}^\infty\abs{\frac{2^{\ell s}}{w(2^{-\ell})}\m_{j_{(k,\ell)},k}\psi_{j_{(k,\ell)}}(D)f\phi_{j_{(k,\ell)}}(D)g}^q\right)^{p/q}\dd x\right\rbrace^{r/p},
\end{equation}
where now $r=\min\lbrace p,1\rbrace$.

Let $N\in\N$. We change variables, apply \eqref{equ8}, the definition of $\norm{\cdot}_{B_{\infty,\infty}^{s+m}(\R^n)}$ and \eqref{equ26} to obtain that
\begin{align}\label{equ37}
    &\left\lbrace\int_ {\R^n}\left(\sum_ {\ell=0}^\infty\abs{\frac{2^{\ell s}}{w(2^{-\ell})}\m_{j_{(k,\ell)},k}\psi_{j_{(k,\ell)}}(D)f\phi_{j_{(k,\ell)}}(D)g}^q\right)^{p/q}\dd x\right\rbrace^{1/p}\nonumber\\
    &=\left\lbrace\int_ {\R^n}\left(\sum_ {\ell=j-k-L}^\infty\abs{\frac{2^{(-j+k+\ell+L)s}}{w(2^{j-k-\ell-L})}\m_{\ell,k}\psi_\ell(D)f\phi_\ell(D)g}^q\right)^{p/q}\dd x\right\rbrace^{1/p}\nonumber\\
    &\lesssim 2^{-js+k(s-N)}\left\lbrace\int_ {\R^n}\left(\sum_ {\ell=0}^\infty\abs{\frac{2^{\ell( s+m)}}{w(2^{j-k-\ell-L})}\psi_\ell(D)f\phi_\ell(D)g}^q\right)^{p/q}\dd x\right\rbrace^{1/p}\nonumber\\
    &\leq 2^{-js+k(s-N)}\left\lbrace\int_ {\R^n}\left(\sum_ {\ell=0}^\infty\abs{\frac{1}{w(2^{j-k-\ell-L})}\phi_\ell(D)g}^q\right)^{p/q}\dd x\right\rbrace^{1/p}\norm{f}_{B_{\infty,\infty}^{s+m}(\R^n)}\nonumber\\
    &\lesssim 2^{-js+k(s-N)}(1+j)^b\left\lbrace\int_ {\R^n}\left(\sum_ {\ell=0}^\infty\abs{\frac{1}{w(2^{-\ell})}\phi_\ell(D)g}^q\right)^{p/q}\dd x\right\rbrace^{1/p}\norm{f}_{B_{\infty,\infty}^{s+m}(\R^n)}.
\end{align}
Now we can use \eqref{equ35}, Proposition \ref{wei-har} and \eqref{equ36} to get that
\[
    \left(\sum_ {\ell=0}^\infty\abs{\frac{1}{w(2^{-\ell})}\phi_\ell(D)g}^q\right)^{1/q}\lesssim\left(\sum_{\ell=0}^\infty \abs{v(2^{-\ell})\tilde{\psi}_\ell(D)g}^q\right)^{1/q},
\]
from where
\begin{align}\label{equ38}
    \left\lbrace\int_ {\R^n}\left(\sum_ {\ell=0}^\infty\abs{\frac{1}{w(2^{-\ell})}\phi_\ell(D)g}^q\right)^{p/q}\dd x\right\rbrace^{1/p}&\lesssim\left\lbrace\int_ {\R^n}\left(\sum_ {\ell=0}^\infty\abs{v(2^{-\ell})\tilde{\psi}_\ell(D)g}^q\right)^{p/q}\dd x\right\rbrace^{1/p}\nonumber\\
    &\lesssim\norm{g}_{F_{p,q}^{0,v}(\R^n)}.
\end{align}
Combining \eqref{aa}, \eqref{equ37} and \eqref{equ38} we obtain that
\begin{align*}\label{equ23}
    \Vert T_\sigma&(f,g) \Vert_{F_ {p,q}^{s,1/w}(\R^n)}\lesssim\left(\sum_{j,k=0}^\infty 2^{j(\tilde{s}-n/2-s)r}(1+j)^{br}2^{-k(N-s)r}\norm{f}_{B_{\infty,\infty}^{s+m}(\R^n)}^r\norm{g}_{F_{p,q}^{0,v}(\R^n)}^r\right)^{1/r}\nonumber\\
    &=\left(\sum_{j=0}^\infty 2^{-j(s-\tilde{s}+n/2)r}(1+j)^{br}\right)^{1/r}\left(\sum_{k=0}^\infty 2^{-k(N-s)r}\right)^{1/r}\norm{f}_{B_{\infty,\infty}^{s+m}(\R^n)}\norm{g}_{F_{p,q}^{0,v}(\R^n)}\nonumber\\
    &\lesssim\norm{f}_{B_{\infty,\infty}^{s+m}(\R^n)}\norm{g}_{F_{p,q}^{0,v}(\R^n)},
\end{align*}
where the series are finite choosing $N$ large enough. 
\subsubsection{Case $p=\infty$}  Proceeding as in the case \ref{finite_case}, we obtain \eqref{eqi1} and \eqref{eqi2}. A change variables, \eqref{equ8}, the definition of $\norm{\cdot}_{B_{\infty,\infty}^{s+m}(\R^n)}$ and \eqref{equ26} yield
\begin{align*}
    &\left(\frac{1}{\abs{Q}}\int_Q\sum_{\ell=-\log_2 \ell(Q)}^\infty\frac{2^{s\ell q}}{w(2^{-\ell})^q}\abs{\m_{j_{(k,\ell)},k}\psi_{j_{(k,\ell)}}(D)f\phi_{j_{(k,\ell)}}(D)g}^q \dd x\right)^{1/q}\nonumber\\
    &=\left(\frac{1}{\abs{Q}}\int_Q\sum_{\ell=j-k-\log_2 \ell(Q)-L}^\infty\frac{2^{s(\ell-j+k+L) q}}{w(2^{-\ell-k-L+j})^q}\abs{\m_{\ell,k}\psi_{\ell}(D)f\phi_{\ell}(D)g}^q \dd x\right)^{1/q}\nonumber\\
    &\lesssim 2^{-js-k(N-s)}\left(\frac{1}{\abs{Q}}\int_Q\sum_{\ell=(j-k-\log_2 \ell(Q)-L)_+}^\infty\frac{1}{w(2^{-\ell-k-L+j})^q}\abs{\phi_{\ell}(D)g}^q \dd x\right)^{1/q}\norm{f}_{B_{\infty,\infty}^{s+m}(\R^n)}\nonumber\\
    &\lesssim 2^{-js-k(N-s)}(1+j)^b\left(\frac{1}{\abs{Q}}\int_Q\sum_{\ell=(j-k-L-\log_2 \ell(Q))_+}^\infty \frac{1}{w(2^{-\ell})^q}\abs{\phi_{\ell}(D)g}^q \dd x\right)^{1/q}\norm{f}_{B_{\infty,\infty}^{s+m}(\R^n)}.
\end{align*}
By using Proposition \ref{wei-har}, \eqref{equ35} and the monotonicity of the integral,  the third term above, within brakets, is bounded by 
\[
    \left(\frac{1}{\abs{Q}}\int_Q\sum_{\ell=(j-k-L-\log_2 \ell(Q))_+}^\infty {v(2^{-\ell})^q}\abs{\tilde{\psi}_{\ell}(D)g}^q \dd x\right)^{1/q}.
\]
Following the argument for studying \eqref{serrahima}, with minor changes, one obtains that this term is bounded by 
\begin{equation*}
    2^{n(-j+k+L)_+/q} \norm{g}_{F_{\infty,q}^{0,v}}.
\end{equation*}
Combining this with the estimates above, we obtain 
\begin{equation}\label{eqi10}
    \begin{split}
        &\sup_{\substack{Q\in\mathcal{D} \\ \ell(Q)\leq 1}}\left(\frac{1}{\abs{Q}}\int_Q\sum_{\ell=-\log_2 \ell(Q)}^\infty\frac{2^{s\ell q}}{w(2^{-\ell})^q}\abs{\varphi_\ell(D)T_\sigma(f,g)(x)}^q \dd x\right)^{1/q}\\
        &\quad\lesssim c(L,r,s,N)\norm{f}_{B_{\infty,\infty}^{s+m}(\R^n)}\norm{g}_{F_{\infty,q}^{0,v}}
    \end{split}    
\end{equation}
where $c(L,r,s,N)^r$ as in \eqref{tort}.

Now it is left to check that 
\begin{equation}\label{eqi11}
    \norm{\varphi_0(D)T_\sigma(f,g)}_\infty\lesssim\norm{f}_{B_{\infty,\infty}^{s+m}(\R^n)}\norm{g}_{F_{\infty,q}^{0,v}},
\end{equation}
which combined with \eqref{eqi10} would give the statement in \ref{equ25}.

To this end, we write $\sigma$ as in \eqref{main_sym} and we see that
\[
    \varphi_0(D)T_\sigma(f,g)(x)=\sum_{j=0}^\infty\varphi_0(D)[\m_j\psi_j(D)f\phi_j(D)g](x).
\]
Let us take $\tilde{s}\in\R$ such that $s>\tilde{s}-n/2>\tau_{p,q}$ and apply Lemma \ref{Lplq}\ref{casopfinita} to get that
\begin{equation}\label{eqi12}
    \norm{\varphi_0(D)T_\sigma(f,g)}_\infty\lesssim\sum_{j=0}^\infty\norm{\varphi_0(2^j\cdot)}_{L^2_{\tilde{s}}(\R^n)}\norm{\m_j\psi_j(D)f\phi_j(D)g}_\infty.
\end{equation}
Next we use \eqref{equ13}, \eqref{equ4} and the definition of $\norm{\cdot}_{B_{\infty,\infty}^{s+m}}$ to obtain that
\begin{equation}
    \begin{split}
        &\sum_{j=0}^\infty\norm{\varphi_0(2^j\cdot)}_{L^2_{\tilde{s}}(\R^n)}\norm{\m_j\psi_j(D)f\phi_j(D)g}_\infty\\
        &\quad\lesssim\sum_{j=0}^\infty 2^{j(\tilde{s}-n/2)}2^{jm}\norm{\psi_j(D)f}_\infty\norm{\phi_j(D)g}_\infty\\
        &\quad=\left(\sum_{j=0}^\infty 2^{j(\tilde{s}-n/2)}2^{-js}\norm{\phi_j(D)g}_\infty\right)\norm{f}_{B_{\infty,\infty}^{s+m}}.
    \end{split}    
\end{equation}
Now we use \eqref{equ35} to get that
\begin{equation}
    \norm{\phi_j(D)g}_\infty\leq\sum_{k=0}^j\norm{\tilde{\psi}_k(D)g}_\infty\leq\norm{g}_{F_{\infty,q}^{0,v}}+\sum_{k=1}^j\norm{\tilde{\psi}_k(D)g}_\infty.
\end{equation}
Let $\lbrace Q_i\rbrace_{i\in\N}\subseteq\mathcal{D}$ be a dyadic decomposition of $\R^n$ such that $\ell(Q_i)=1$ for all $i\in\N$. Let $\tilde{\phi}$ be a Schwartz function and set $\tilde{\phi}_k:=\tilde{\phi}(2^{-k}\cdot)$, in such a way that $\tilde{\phi}_k$ is identically 1 in the support of $\tilde{\psi}_k$. For $1\leq k\leq j$, we use Young's inequality, the definition of $\norm{\cdot}_{F_{\infty,q}^{0}}$ and the fact that $\inf_{t>0} v(t)>0$ to obtain that
\begin{equation}
    \begin{split}
        \norm{\tilde{\psi}_k(D)g}_\infty&=\sup_{i\in\N}\norm{\tilde{\psi}_k(D)g}_{L^\infty(Q_i)}\\
        &=\sup_{i\in\N}\norm{\tilde{\phi}_k(D)[\tilde{\psi}_k(D)g]}_{L^\infty(Q_i)}\\
        &\leq\norm{\mathcal{F}^{-1}(\tilde{\phi}_k)}_{L^{q'}(\R^n)}\sup_{i\in\N}\norm{\tilde{\psi}_k(D)g}_{L^q(Q_i)}\\
        &\lesssim \norm{\mathcal{F}^{-1}(\tilde{\phi}_k)}_{L^{q'}(\R^n)}\norm{g}_{F_{\infty,q}^{0}(\R^n)}\\
        &\lesssim \norm{\mathcal{F}^{-1}(\tilde{\phi}_k)}_{L^{q'}(\R^n)}\norm{g}_{F_{\infty,q}^{0,v}(\R^n)}.
    \end{split}
\end{equation}
In addition, we notice that
\begin{equation}\label{eqi13}
    \norm{\mathcal{F}^{-1}(\tilde{\phi}_k)}_{L^{q'}(\R^n)}=2^{-kn/q'}\norm{\mathcal{F}^{-1}(\tilde\phi)}_{L^{q'}(\R^n)}\approx 2^{-kn/q'}.
\end{equation}
Combining equations \eqref{eqi12}-\eqref{eqi13} yields
\begin{align*}
    \norm{\varphi_0(D)T_\sigma(f,g)}_\infty&\lesssim\left(\sum_{j=0}^\infty 2^{-j(s-\tilde{s}+n/2)}\left(1+\sum_{k=1}^j 2^{-kn/q'}\right)\right)\norm{f}_{B_{\infty,\infty}^{s+m}}\norm{g}_{F_{\infty,q}^{0,v}(\R^n)}\\
    &\lesssim\left(\sum_{j=0}^\infty 2^{-j(s-\tilde{s}+n/2)}\right)\norm{f}_{B_{\infty,\infty}^{s+m}}\norm{g}_{F_{\infty,q}^{0,v}(\R^n)}\\
    &\lesssim\norm{f}_{B_{\infty,\infty}^{s+m}}\norm{g}_{F_{\infty,q}^{0,v}(\R^n)},
\end{align*}
and hence \eqref{eqi11} is satisfied. This completes the proof of the theorem.
\begin{rem}
    In part Theorem \ref{main_thm} \ref{equ31} we could replace $\norm{g}_{X_w(\R^n)}$ by the norm of $g$ in a Besov space of generalised smoothness, $\norm{g}_{B_{\infty,\infty}^{0,u}(\R^n)}$, when $u$ is an admissible weight such that
    \begin{equation}\label{wu}
        w(t)=1+\int_{\min\lbrace t,1\rbrace}^1 \frac{\dd s}{su(s)},\quad t>0.
    \end{equation}
    Indeed, the estimate in \eqref{equ7} could be replaced by
    \[
        \norm{\phi(2^{-\ell}D)g}_\infty\leq \brkt{1+\int_{2^{-\ell}}^1 \frac{\dd s}{su(s)}} \norm{g}_{B_{\infty,\infty}^{0,u}(\R^n)}=w(2^{-\ell})\norm{g}_{B_{\infty,\infty}^{0,u}(\R^n)}.
    \]
    We notice also that for $w$ as in \eqref{wu} it holds that
    \[
        \BMO(\R^n)\cap B_{\infty,\infty}^{0,u}(\R^n)\subset X_w(\R^n).
    \]
\end{rem}

\section*{Acknowledgements}
We thank the anonymous referees for their constructive comments and suggestions that have help improving he final version of this paper.


\begin{bibdiv}
\begin{biblist}
\bib{Paper1}{article}{
   author={Arias, Sergi},
   author={Rodr\'{\i}guez-L\'{o}pez, Salvador},
   title={Some endpoint estimates for bilinear Coifman-Meyer multipliers},
   journal={J. Math. Anal. Appl.},
   volume={498},
   date={2021},
   number={2},
   pages={124972, 27},
%
%
%
}
\bib{Bennet}{article}{
   author={Bennett, Grahame},
   title={Some elementary inequalities. III},
   journal={Quart. J. Math. Oxford Ser. (2)},
   volume={42},
   date={1991},
   number={166},
   pages={149--174},
   %
   %
   %
}
\bib{ben-nah-tor}{article}{
   author={B\'{e}nyi, \'{A}rp\'{a}d},
   author={Nahmod, Andrea R.},
   author={Torres, Rodolfo H.},
   title={Sobolev space estimates and symbolic calculus for bilinear
   pseudodifferential operators},
   journal={J. Geom. Anal.},
   volume={16},
   date={2006},
   number={3},
   pages={431--453},
   %
   %
   %
}
\bib{ben-tor}{article}{
   author={B\'{e}nyi, \'{A}rp\'{a}d},
   author={Torres, Rodolfo H.},
   title={Symbolic calculus and the transposes of bilinear
   pseudodifferential operators},
   journal={Comm. Partial Differential Equations},
   volume={28},
   date={2003},
   number={5-6},
   pages={1161--1181},
   %
   %
   %
}
\bib{Bru-Nai1}{article}{
   author={Brummer, Joshua},
   author={Naibo, Virginia},
   title={Bilinear operators with homogeneous symbols, smooth molecules, and
   Kato-Ponce inequalities},
   journal={Proc. Amer. Math. Soc.},
   volume={146},
   date={2018},
   number={3},
   pages={1217--1230},
   issn={0002-9939},
%
%
}
\bib{Bru-Nai2}{article}{
   author={Brummer, Joshua},
   author={Naibo, Virginia},
   title={Weighted fractional Leibniz-type rules for bilinear multiplier
   operators},
   journal={Potential Anal.},
   volume={51},
   date={2019},
   number={1},
   pages={71--99},
   issn={0926-2601},
%
%
}
\bib{cae-mou}{article}{
   author={Caetano, Ant\'{o}nio M.},
   author={Moura, Susana D.},
   title={Local growth envelopes of spaces of generalised smoothness: the
   subcritical case},
   journal={Math. Nachr.},
   volume={273},
   date={2004},
   pages={43--57},
%
}
\bib{dominguez2018function}{article}{
   author={Dom{\'\i}nguez, Oscar},
   author={Tikhonov, Sergey},
   title={Function spaces of logarithmic smoothness: embeddings and characterizations},
  journal={arXiv preprint arXiv:1811.06399},
  year={2018}
}

\bib{Grafakos-Oh}{article}{
   author={Grafakos, Loukas},
   author={Oh, Seungly},
   title={The Kato-Ponce inequality},
   journal={Comm. Partial Differential Equations},
   volume={39},
   date={2014},
   number={6},
   pages={1128--1157},
   issn={0360-5302},
%
%
}
\bib{Gra-Mal_Nai}{article}{
   author={Grafakos, Loukas},
   author={Maldonado, Diego},
   author={Naibo, Virginia},
   title={A remark on an endpoint Kato-Ponce inequality},
   journal={Differential Integral Equations},
   volume={27},
   date={2014},
   number={5-6},
   pages={415--424},
   issn={0893-4983},
%
}
\bib{Goldberg}{article}{
   author={Goldberg, David},
   title={A local version of real Hardy spaces},
   journal={Duke Math. J.},
   volume={46},
   date={1979},
   number={1},
   pages={27--42},
%
   %
}
\bib{koe-tom}{article}{
   author={Koezuka, Kimitaka},
   author={Tomita, Naohito},
   title={Bilinear pseudodifferential operators with symbols in
   $BS^m_{1,1}$ on Triebel-Lizorkin spaces},
   journal={J. Fourier Anal. Appl.},
   volume={24},
   date={2018},
   number={1},
   pages={309--319},
%
   %
   %
}
\bib{KuMaPe}{book}{
   author={Kufner, Alois},
   author={Maligranda, Lech},
   author={Persson, Lars-Erik},
   title={The Hardy inequality},
   note={About its history and some related results},
   publisher={Vydavatelsk\'{y} Servis, Plze\v{n}},
   date={2007},
   pages={162},
   %
   %
}
\bib{mou}{article}{
   author={Moura, Susana},
   title={Function spaces of generalised smoothness},
   journal={Dissertationes Math. (Rozprawy Mat.)},
   volume={398},
   date={2001},
   pages={88},
 %
}
\bib{Naibo}{article}{
   author={Naibo, Virginia},
   title={On the bilinear H\"{o}rmander classes in the scales of
   Triebel-Lizorkin and Besov spaces},
   journal={J. Fourier Anal. Appl.},
   volume={21},
   date={2015},
   number={5},
   pages={1077--1104},
   issn={1069-5869},
%
%
}
\bib{Naibo-Thomson}{article}{
   author={Naibo, Virginia},
   author={Thomson, Alexander},
   title={Coifman-Meyer multipliers: Leibniz-type rules and applications to
   scattering of solutions to PDEs},
   journal={Trans. Amer. Math. Soc.},
   volume={372},
   date={2019},
   number={8},
   pages={5453--5481},
   issn={0002-9947},
%
%
}
\bib{Park2}{article}{
   author={Park, Bae Jun},
   title={Fourier multipliers on a vector-valued function space},
   journal={Constr. Approx.},
   volume={55},
   date={2022},
   number={2},
   pages={705--741},
   issn={0176-4276},
   %
   %
}
\bib{Park}{article}{
   author={Park, Bae Jun},
   title={Equivalence of (quasi-)norms on a vector-valued function space and
   its applications to multilinear operators},
   journal={Indiana Univ. Math. J.},
   volume={70},
   date={2021},
   number={5},
   pages={1677--1716},
   issn={0022-2518},
   %
   %
}
\bib{rod-sta}{article}{
   author={Rodr\'{\i}guez-L\'{o}pez, Salvador},
   author={Staubach, Wolfgang},
   title={Some endpoint estimates for bilinear paraproducts and
   applications},
   journal={J. Math. Anal. Appl.},
   volume={421},
   date={2015},
   number={2},
   pages={1021--1041},
   %
   %
   %
}
\bib{Runst-Sickel}{book}{
   author={Runst, Thomas},
   author={Sickel, Winfried},
   title={Sobolev spaces of fractional order, Nemytskij operators, and
   nonlinear partial differential equations},
   series={De Gruyter Series in Nonlinear Analysis and Applications},
   volume={3},
   publisher={Walter de Gruyter \& Co., Berlin},
   date={1996},
   pages={x+547},
   isbn={3-11-015113-8},
   %
   %
}
\bib{Stein}{book}{
   author={Stein, Elias M.},
   title={Harmonic analysis: real-variable methods, orthogonality, and
   oscillatory integrals},
   series={Princeton Mathematical Series},
   volume={43},
   note={With the assistance of Timothy S. Murphy;
   Monographs in Harmonic Analysis, III},
   publisher={Princeton University Press, Princeton, NJ},
   date={1993},
   pages={xiv+695},
   %
   %
}
\bib{Trie83}{book}{
   author={Triebel, Hans},
   title={Theory of function spaces},
   series={Monographs in Mathematics},
   volume={78},
   publisher={Birkh\"{a}user Verlag, Basel},
   date={1983},
   pages={284},
 %
   %
   %
}
\end{biblist}
\end{bibdiv}
\end{document}